%% file: HAMSv2.tex
\documentclass[10pt]{amsart}

\input{preamble}

\title{Higher Associativity of Moore Spectra}

\author{Prasit~Bhattacharya}\address{University of Virginia}\email{pb9wh@virginia.edu}

\thanks{}

\begin{document}

\begin{abstract}
 The Moore spectrum $\mathrm{M}_p(i)$ is the cofiber of the $p^{i}$ map on the sphere spectrum. For a fixed $p$ and $n$, we find a lower bound on $i$ for which $\mathrm{M}_p(i)$ is guaranteed to be $n$-fold associative. This bound depends on the stable homotopy groups of spheres.\end{abstract}

\maketitle	


\section{Introduction} \label{Sec:intro}
\input{SECTION1-Introduction.tex}

\subsection*{Acknowledgements}
It is a  privilege to have discussed this problem with late Mark Mahowald and would like to thank him for many insightful comments regarding this project. I have  also benefitted from discussions with Vigliek Angleveit, Tyler Lawson, Akhil
Mathew and Haynes Miller. Above all, I would like to thank my adviser Michael Mandell for introducing me to this project as well as helping me out at various stages of this project. This project would have been impossible without his support, encouragement and insight. I would also like to thank Mathematical Science Research Institute for its hospitality during which part of this research was done.

%
%
\section{ Background} 
 \label{Sec:background}
  \input{SECTION2-background.tex}

\section{Moore spectra as Thom spectra}
 \label{Sec:Moorethom}
 \input{SECTION3-thomspectrum.tex}
\section{Proof of the \Cref{main:est1}} 
\label{Sec:miantheorem}
\input{SECTION4-Maintheorem.tex}



\bibliographystyle{amsalpha}
\bibliography{HAMSv2}
\end{document}

%% file: preamble.tex
\usepackage{mathtools}
\usepackage{amsthm}
\usepackage{amssymb}
\usepackage[all]{xy}
\usepackage[nobottomtitles*]{titlesec}
\usepackage{titletoc}
\usepackage{bbm}
\usepackage{upgreek}
\usepackage{ marvosym }
\usepackage{lscape}
\usepackage{ifluatex}
\usepackage[pdfencoding=auto, psdextra]{hyperref}
\usepackage[noabbrev,nameinlink]{cleveref}
\usepackage{enumerate}
\usepackage{soul}
\usepackage{stmaryrd}

\usepackage{verbatim}
\usepackage{tikz-cd}
\makeatletter
\tikzcdset{
    cong/.style={"\cong" {sloped, description, yshift=0pt,#1}, phantom},
    simeq/.style={"\simeq" {sloped, description, yshift=0pt,#1}, phantom},
    snake/.style={
        out= east, in=west,
        to path={
            \pgfextra{
                \pgfextractx{\pgf@xa}{\pgfpointanchor{\tikztostart}{east}}
                \pgfextractx{\pgf@xb}{\pgfpointanchor{\tikztotarget}{west}}
                \pgfextracty{\pgf@ya}{\pgfpointanchor{\tikztostart}{center}}
                \pgfextracty{\pgf@yb}{\pgfpointanchor{\tikztotarget}{center}}
                \edef\tikzstartx{\the\pgf@xa}
                \edef\tikzendx{\the\pgf@xb}
                \edef\midy{\the\dimexpr0.5\dimexpr\pgf@ya\relax +0.5\dimexpr\pgf@yb\relax}
            }
            to[in=0,out=180,looseness=0.5] (\tikzstartx,\midy)
            -| ([xshift=-2ex]\tikztotarget.west)
            -- (\tikztotarget)}
    }
}
\makeatother
\usepackage{tikzcdintertext}
\usepackage{etoolbox}
\usepackage{xpatch}

\crefname{diagram}{diagram}{diagrams}
\creflabelformat{diagram}{#2(#1)#3}
\crefname{sseq}{}{}
\crefname{equation}{}{}
\crefname{equation-b}{Equation}{Equations}
\creflabelformat{sseq}{#2(#1)#3}
\usepackage[normalem]{ulem}
\newcommand{\stkout}[1]{\ifmmode\text{\sout{\ensuremath{#1}}}\else\sout{#1}\fi}

\definecolor{darkspringgreen}{rgb}{0.09, 0.45, 0.27}

\makeatletter

\pretocmd{\maketitle}{%
    \edef\@title{\unexpanded{\protect\Large}\unexpanded\expandafter{\@title}}%
    \edef\authors{\unexpanded{\protect\normalsize}\unexpanded\expandafter{\authors}}%
}{}{\error}

\def\addlabeltolink#1{{\addlabeltolink@#1}}
\def\addlabeltolink@#1#2#3#4{#1{#2}{#3}{\thecontentslabel. #4}}
\contentsmargin{-1.5em}
\titlecontents{section}[1em]{}{\addlabeltolink}{\hspace*{-1em}}{
\titlerule*[1pc]{.}\hphantom{11.}\llap{\thecontentspage}}

\titleformat{\section}[block]{\centering\bfseries\Large}{\thetitle. }{0pt}{}
\titlespacing{\section}{0pt}{*4}{*2}
\titleformat{\subsection}[hang]{\bfseries}{\thetitle. }{0pt}{}
\titlespacing{\subsection}{0pt}{*2}{*1.5}

\def\@secnumpunct{. }
\xpatchcmd{\proof}{\topsep6\p@\@plus6\p@\relax}{\topsep0pt\relax}{}{\error}
\makeatother
\usepackage[parfill]{parskip}

\allowdisplaybreaks[1]

\newcommand\Einfty{\mathbb{E}_{\infty}}

\newcommand\Ainfty{\A_{\infty}}

\newcommand{\M}[2]{\mr{M}_{#1}\hspace{-2.5pt}\left(#2 \right)}

\newcommand{\mr}[1]{\mathrm{#1}}

\newcommand{\GG}{\mathbb{G}}

\newcommand{\Ab}{\mathcal{A}\mr{b}}

\def\S{\mathbb{S}}

\newcommand{\CP}{\mathbb{CP}}

\newcommand{\sph}{\mathbb{S}}
\newcommand{\csph}{\mathcal{S}_{\sph}}
\newcommand{\disk}{\mathcal{D}_{\sph}}

\newcommand{\R}{\mr{R}}

\newcommand\E{\mr{E}}

\newcommand{\Func}{\mr{Func}}

\newcommand{\K}{\mr{K}}

\newcommand{\A}{\mathbb{A}}

\newcommand{\iso}{\cong}
\newcommand{\Cl}{\mr{C}}
\newcommand{\brk}[1]{\hspace{-.5pt}\langle #1\rangle }
\newcommand{\tsph}{\mathcal{S}}
\newcommand{\tdisk}{\mathcal{D}}
\newcommand{\Prin}{\mr{Prin}}
\newcommand{\csphp}{\mathcal{S}_{\sph_p}}
\DeclareMathOperator{\padic}{\hat{\mathbb{Z}}}
\newcommand*{\rom}[1]{\expandafter\@slowromancap\romannumeral #1@} \makeatother

\hypersetup{%
  bookmarksnumbered=true,%
  bookmarks=true,%
  colorlinks=true,%
  pdfnewwindow=true,%
  pdfstartview=FitBH%
}

\hypersetup{
    colorlinks,
    linkcolor={red!30!brown!70!black},
    citecolor={blue!70!black},
    urlcolor={blue!80!black}
}

\def\@url#1{{\tt\def~{\lower3.5pt\hbox{\char'176}}\def\_{\char'137}#1}}

\makeatletter
\let\c@lemma\c@theorem
\makeatother

\newtheorem{thm}[equation]{Theorem}
\newtheorem{main}[equation]{Main Theorem}
\newtheorem{cor}[equation]{Corollary}
\newtheorem{lem}[equation]{Lemma}

\newtheorem{conj}[equation]{Conjecture}

\theoremstyle{definition}
\newtheorem{defn}[equation]{Definition}

\newtheorem{ex}[equation]{Example}

\newtheorem{rmk}[equation]{Remark}
\newtheorem{notn}[equation]{Notation}

\newtheorem*{thm*}{Theorem}
\newtheorem*{cor*}{Corollary}
\newtheorem*{lem*}{Lemma}
\newtheorem*{prop*}{Proposition}
\newtheorem*{not*}{Notation}
\newtheorem*{guess*}{Guess}

\newtheorem*{defn*}{Definition}
\newtheorem*{ex*}{Example}
\newtheorem*{exs*}{Examples}
\newtheorem*{rmk*}{Remark}
\newtheorem*{claim*}{Claim}
\newtheorem*{exer*}{Exercise}

\numberwithin{equation}{section}
\numberwithin{figure}{section}

\makeatletter
\let\c@lem=\c@thm
\let\c@cor=\c@thm
\let\c@prop=\c@thm
\let\c@lem=\c@thm
\let\c@ex=\c@thm
\let\c@exs=\c@thm
\let\c@obs=\c@thm
\let\c@rmk=\c@thm
\let\c@perthm=\c@thm
\let\c@conjtel=\c@thm
\let\c@exmps=\c@thm
\let\c@rem=\c@thm
\let\c@question=\c@thm
\let\c@warn=\c@thm
\let\c@claim=\c@thm
\let\c@quest=\c@thm
\let\c@notation=\c@thm
\let\c@note=\c@thm
\let\c@conjtel=\c@thm
\let\c@gue=\c@thm
\let\c@goal=\c@thm
\makeatother

\DeclareMathOperator\AHSS{AHSS}

\DeclareMathOperator*{\colim}{colim}

\DeclareMathOperator{\Sp}{\mathcal{S}p}
\DeclareMathOperator{\Top}{\mathcal{T}op}
\DeclareMathOperator{\TTop}{\mathcal{M}}

\DeclareMathOperator{\GL}{GL_1}
\DeclareMathOperator{\SL}{SL_1}
\DeclareMathOperator{\BGL}{BGL_1}

\newcommand{\T}{\mathbb{T}}

\renewcommand\H{\mathrm{H}}

\def\makecommands#1#2#3{
    \bgroup
    \def\tempcmdname##1{#1}
    \def\tempcmdbody##1{#2}
    \def\\##1{\expandafter\xdef\csname\tempcmdname{##1}\endcsname{\unexpanded\expandafter{\tempcmdbody{##1}}}}
    #3
    \egroup
}
\def\upperalphabet{\\A\\B\\C\\D\\E\\F\\G\\H\\I\\J\\K\\L\\M\\N\\O\\P\\Q\\R\\S\\T\\U\\V\\W\\X\\Y\\Z}
\def\loweralphabet{\\a\\b\\c\\d\\e\\f\\g\\h\\i\\j\\k\\l\\m\\n\\o\\p\\q\\r\\s\\t\\u\\v\\w\\x\\y\\z}
\def\lowergreekalphabet{\\\alpha\\\beta\\\gamma\\\delta\\\epsilon\\\zeta\\\eta\\\theta\\\kappa\\\lambda\\\mu\\\nu
    \\\xi\\\pi\\\rho\\\sigma\\\tau\\\upsilon\\\psi\\\chi\\\phi\\\omega}

\makecommands{#1mr}{\mathrm{#1}}{\upperalphabet}
\makecommands{#1cal}{\mathcal{#1}}{\upperalphabet}
\makecommands{#1#1}{\mathbb{#1}}{\upperalphabet}
\makecommands{#1bar}{\overline{#1}}{\upperalphabet\loweralphabet}
\makecommands{#1twee}{\widetilde{#1}}{\upperalphabet\loweralphabet}
\makecommands{sf#1}{{\sf #1}}{\upperalphabet\loweralphabet}
\makeatletter
\makecommands{\expandafter\@gobble\string#1bar}{\overline{#1}}{\lowergreekalphabet}
\makecommands{\expandafter\@gobble\string#1twee}{\widetilde{#1}}{\lowergreekalphabet}
\makeatother

\newcommand{\sma}{\wedge}

\definecolor{limegreen}{rgb}{0.2, 0.8, 0.2}
\definecolor{darkmagenta}{rgb}{0.55, 0.0, 0.55}
\definecolor{lavenderrose}{rgb}{0.91, 0.33, 0.5}
\definecolor{goldenpoppy}{rgb}{0.99, 0.76, 0.0}
\definecolor{seagreen}{rgb}{0.11, 0.35, 0.02}
\definecolor{maroon}{RGB}{128,0,0}
\definecolor{darkviolet}{RGB}{148,0,211}

\makeatletter

\def\tikzcdequalsignoffset{0.1em}

\def\findedgesourcetarget#1#2{
    \let\sourcecoordinate\pgfutil@empty
    \ifx\tikzcd@startanchor\pgfutil@empty 
        \def\tempa{\pgfpointanchor{#1}{center}}
    \else
        \edef\tempa{\noexpand\pgfpointanchor{#1}{\expandafter\@gobble\tikzcd@startanchor}} 
        \let\sourcecoordinate\tempa
    \fi
    \ifx\tikzcd@endanchor\pgfutil@empty 
        \def\tempb{\pgfpointshapeborder{#2}{\tempa}}
    \else
        \edef\tempb{\noexpand\pgfpointanchor{#2}{\expandafter\@gobble\tikzcd@endanchor}}
    \fi
    \let\targetcoordinate\tempb
    \ifx\sourcecoordinate\pgfutil@empty%
        \def\sourcecoordinate{\pgfpointshapeborder{#1}{\tempb}}%
    \fi
}

\tikzset{/tikz/commutative diagrams/equal/.style=equals,
    /tikz/commutative diagrams/equals/.style = {
    -,
    to path={\pgfextra{
        \findedgesourcetarget{\tikzcd@ar@start}{\tikzcd@ar@target} 
        \ifx\tikzcd@startanchor\pgfutil@empty
            \def\tikzcd@startanchor{.center}
        \fi
        \ifx\tikzcd@endanchor\pgfutil@empty
            \def\tikzcd@endanchor{.center}
        \fi
        \pgfmathanglebetweenpoints{\pgfpointanchor{\tikzcd@ar@start}{\expandafter\@gobble\tikzcd@startanchor}}{\pgfpointanchor{\tikzcd@ar@target}{\expandafter\@gobble\tikzcd@endanchor}}
        \pgftransformrotate{\pgfmathresult}
        \pgfpathmoveto{\pgfpointadd{\sourcecoordinate}{\pgfpoint{0}{\tikzcdequalsignoffset}}}
        \pgfpathlineto{\pgfpointadd{\targetcoordinate}{\pgfpoint{0}{\tikzcdequalsignoffset}}}
        \pgfpathmoveto{\pgfpointadd{\sourcecoordinate}{\pgfpoint{0}{-\tikzcdequalsignoffset}}}
        \pgfpathlineto{\pgfpointadd{\targetcoordinate}{\pgfpoint{0}{-\tikzcdequalsignoffset}}}
        \pgfusepath{draw}
}}}}

\makeatother

%% file: SECTION1-Introduction.tex
Let $(\Sp, \sma, \sph)$ denote the category of $\sph$-modules of \cite{EKMM}, a modern point-set  category of spectra, which is closed symmetric monoidal. In this category, an associative monoid is precisely an $\Ainfty$-ring spectrum and a commutative monoid is precisely an $\Einfty$-ring spectrum. The sphere spectrum $\sph^0$, being the unit, is a perfectly good $\Einfty$-ring spectrum. The $i$-th mod $p$ Moore spectrum, denoted by $\M{p}{i}$, is the cofiber of the multiplication by $p^i$ map 
\[  \begin{tikzcd}
 \sph \rar[" p^i"]  & \sph \rar& \M{p}{i}.
 \end{tikzcd} \]
 
The category $(\Sp, \sma, \sph)$ can be compared to the symmetric monoidal category of Abelian groups $(\Ab, \otimes, \mathbb{Z})$ via the monoidal functor 
\[
\begin{tikzcd}
 \pi_0(-) : (\Sp, \sma, \sph) \rar & (\Ab, \otimes, \mathbb{Z})
 \end{tikzcd} \]
 which assigns a spectrum its zeroth stable homotopy group. Under this comparison $\M{p}{i}$ is analogous to $\mathbb{Z}/p^{i}$, the quotient of  the multiplication by $p^i$ map on the unit $\mathbb{Z}$.   While  $\mathbb{Z}/p^{i}$ inherits the commutative and the associative structure from $\mathbb{Z}$, $\M{p}{i}$ does not inherit the $\Einfty$-structure from $\sph$ (see  \cite[Remark 4.3]{MNN} for a proof).  In fact, Mark Mahowald conjectured that $\M{p}{i}$ does not  even admit an $\Ainfty$-structure (see \Cref{conj:mah}). The conjecture holds true in the case when $i =1$. 

\begin{ex} \label{ex:even}
 The mod $2$ Moore spectrum $\M{2}{1}$ does not admit a multiplication. The proof is as follows. If $\M{2}{1}$ does admit a unital multiplication  
\[ \begin{tikzcd}
 \mu: \M{2}{1} \sma \M{2}{1} \rar& \M{2}{1},
 \end{tikzcd}
 \]
then the map $\mu$ splits the cofiber sequence 
\[ 
\begin{tikzcd}
\M{2}{1} \rar["2"] & \M{2}{1} \rar & \arrow[bend right =30, dashed]{l}{\mu} \M{2}{1} \sma \M{2}{1} \rar & \Sigma \M{2}{1} \rar &\dots 
\end{tikzcd} 
\]
resulting in an equivalence 
\[ \M{2}{1} \sma \M{2}{1} \simeq \M{2}{1} \vee \Sigma \M{2}{1}.\]
However, the above equivalence is a contradiction to the fact that the cohomology of $\M{2}{1} \sma \M{2}{1}$  does not split as a module over the Steenrod algebra.   
\end{ex} 
\begin{ex} \label{ex:odd} If $p$ is odd, the $\M{p}{1}$ an $\mathbb{A}_{p-1}$-algebra structure (see~\Cref{ex:oddmoore}) but not an $\mathbb{A}_p$-algebra structure. The proof is obtained by combining the work of Toda~\cite{Toda}, Kochman~\cite{Ko} and Kraines~\cite{Kr} as explained in \cite[Example~3.3]{VKT1}.
\end{ex}

Following \Cref{conj:mah}, \Cref{ex:even} and \Cref{ex:odd}, it is natural to ask `for what values of $n$ does $\M{p}{i}$ admit an $\A_n$-structure when $i>1$?'

Before the results in this paper, hardly anything was known about the $\mathbb{A}_n$-structures of $\M{p}{i}$ for $i>1$. The only result dates back to 1982 when Oka \cite[Theorem 2]{Oka} proved that $\M{2}{i}$ admits an $\mathbb{A}_3$-structure (i.e. a homotopy associative multiplication) for $i \geq 2$. In this paper, we prove that: 
\begin{main}\label{main:est1}Fix a prime $p$ and an integer $n > 1$. Define the function $o_p(n)$ as
\[ o_p(n) = \# \lbrace k:\text{ $k \leq 2n-3$, $k$ odd, and  $p$-torsion of $\pi_k(\sph^0)$ is nonzero}\rbrace. \]
 When $p$ is odd, $\M{p}{i}$ admits an $\A_n$-algebra structure if $i > o_p(n)$. When $p=2$, $\M{2}{i}$ admits an $\A_n$-algebra structure if $i > o_2(n)+1$. 
\end{main}

 \begin{table}[h] 
 \[ \begin{tabular}{| l | c  c  c | l | c c c |} 
  \hline\hline
   $n$ & $o_2(n)$ & $o_3(n)$ & $o_5(n)$ & $n$ & $o_2(n)$ & $o_3(n)$ & $o_5(n)$\\
   \hline \hline
  2 & 1 & 0 & 0 & 9 & 6& 5 & 2\\
  3 & 2 & 1 & 0 & 10 & 7& 5& 2\\
  4 & 2 & 1 & 0 & 11 & 8 & 6& 2\\
  5 & 3 & 2 & 1 & 12 & 9 & 6& 2\\
  6 & 4 & 2 & 1 & 13 & 10& 7 & 3\\
  7 & 5 & 3 & 1 & 14 & 11 & 7 & 3\\
  8 & 5 & 4 & 1 & 15 & 12 & 8 & 3\\
  \hline
\end{tabular} \]
 \label{table}
\caption{Values of $o_p(n)$ for $p = 2,3$ and $5$.}
\end{table}
To establish the usefulness of \Cref{main:est1}, we provide Table~\ref{table2} listing few examples  of the highest value of $n$ for which $\M{p}{i}$ admits $\A_n$-structure. The conclusions are made using  \Cref{main:est1}, expect when $p=2$ and $i \in \{ 2,3 \}$. In the case of $\M{2}{2}$ and $\M{2}{3}$, \Cref{ex:moore2} provides a better answer.
\begin{table}[h]
\[ \begin{array}{| l || c | c | c |  c | c | c | c |} 
  \hline\hline
  i \to & 1 &2 &3 & 4 & 5 & 6 & 7   \\
   \hline \hline
  p=2 & \A_1 & \A_4 & \A_4 & \A_4 &\A_5 & \A_6 & \A_8 \\
   p=3 & \A_2 & \A_4 & \A_6 & \A_7 & \A_8 & \A_{10} & \A_{12} \\
   p=5 & \A_4 & \A_8 & \A_{12} & \A_{16} &\A_{20}& \A_{23} & \A_{24}\\
    \hline
\end{array} \] 
\caption{List of highest $n$ for which $\M{p}{i}$ is known to admit an $\A_n$-structure ( Restricted to $i \leq 7$ and $p = 2,3$ and $5$.)} \label{table2}
\end{table}

The key idea behind the proof of \Cref{main:est1} is to obtain $\M{p}{i}$ as a Thom spectrum (see \Cref{cor:Moorethom}). Using a result due to Stasheff \Cref{thm:barasso}, the problem of obtaining  $\A_n$-structure on $\M{p}{i}$ can be reduced to the study of an Atiyah-Hirzebruch spectral sequence \Cref{eqn:AHSS}.

We end the paper with \Cref{conj:J}, which predicts  that the obstruction to $\Ainfty$-structure on $\M{p}{i}$ is in the chromatic layer $1$ of its homotopy groups. 

\begin{rmk}
 For this paper, one can also choose any other modern point-set categories of spectra, such as the symmetric spectra in simplicial sets of \cite{HSS} or the orthogonal spectra of \cite{MMSS}, however, technical adjustments needed at various stages of the paper may vary depending on the chosen category of spectra. 
\end{rmk}
\subsection*{Organization of the paper} In \Cref{Sec:background} has been established to provide the necessary background and discuss finer details that goes into the proof of \Cref{main:est1}. In the process,  we discuss basics of operad theory, Stasheff's $\A_n$-operad,  obstruction theory for $\A_n$-structures (developed by Stasheff), a conjecture of Mahowald for $2$-cell complexes, $\A_n$-maps, homotopy $\A_n$-maps, an obstruction theory for homotopy $\A_n$-maps (\Cref{thm:barasso}) and a rigidification theorem of Boardman and Vogt (\Cref{thm:Anfix}) which converts a homotopy $\A_n$-map into an $\A_n$-map.  The best known $\A_n$-structures on $\M{2}{2}$ and $\M{2}{3}$ are obtained as an application of an obstruction theory (see \Cref{ex:moore2}). 
 
In \Cref{Sec:Moorethom}, we review the general construction of Thom spectra and show that Moore spectra can be realized as Thom spectra. 
 
In \Cref{Sec:miantheorem} we prove \Cref{main:est1}.

%% file: SECTION2-background.tex
\subsection{Operads and operad algebras}

Operads are mathematical objects which encodes algebraic structures in a symmetric monoidal category. 
\begin{defn} \label{defn:operad}
 A topological  \emph{operad} $\mathcal{O}$ is a sequence of spaces $\mathcal{O}(n)$ for $n \geq 0$  together with distinguished element $1 \in \mathcal{O}(1)$ and the data of continuous functions 
\[ \gamma(n;j_1, \dots, j_n): \mathcal{O}(n) \times (\mathcal{O}(j_1)\times \dots \times \mathcal{O}(j_n)) \to \mathcal{O}(j) \]
where $j = j_1 + \dots + j_n$, which satisfy 
\begin{enumerate}[(i)]
\item $\gamma(\gamma(n;j_1, \dots, j_n); i_1, \dots, i_j) = \gamma(n; \gamma(j_1;i_1, \dots, i_{j_1}), \dots, \gamma(j_n; i_{j-j_n+1}, \dots, i_j))$, 
\item $\gamma(1;n)(1,x) = x \text{ and } \gamma(n;1, \dots, 1)(x, 1, \dots, 1) = x.$
\end{enumerate}
A \emph{nonunital operad} is an operad such that $\mathcal{O}(0)= \emptyset$. A \emph{unital operad} is an operad with $\mathcal{O}(0) = *$. 
\end{defn}
\begin{rmk} Operads often come with symmetries, and those that do not are called non-$\Sigma$ operads. Since we mostly  work with non-$\Sigma$ operads, we use the  term `operad' to refer to non-$\Sigma$ operads.
\end{rmk}
\begin{rmk} Let $\circ_{i}$ denote the map
\[ \gamma(n, 1,\dots,1,k,1,\dots,1): \mathcal{O}(n) \times (\mathcal{O}(1)^{\times i-1} \times \mathcal{O}(k) \times \mathcal{O}(1)^{\times n-i}) \longrightarrow \mathcal{O}(n+k -1).\]
 The operations $\{ \circ_i: i \geq 1 \}$ determine and are determined by the operad structure on $\mathcal{O}$ (see \cite[1.7.1]{MSS}).
\end{rmk}
\begin{ex}[Endomorphism operad] \label{ex:endoperad} For every object $X \in \Top$, the \emph{endomorphism   operad} $\mathcal{E}(X)$ is the operad whose  $n$-th space is \[ \mathcal{E}(X)(n) = \Func(X^{\times n}, X)\]
 and the map 
\[ \circ_i: \mathcal{E}(X)(n) \times \mathcal{E}(X)(m) \longrightarrow \mathcal{E}(X)(n+m-1)\]
 sends $(f,g)$ to the composite
\[ 
\begin{tikzcd}
X^{\times (n+m-1)} \arrow{rrr}{\mathbbm{1}_X^{\times (i-1)} \times g \times \mathbbm{1}_X^{\times (n-i)}}&&& X^{\times n} \rar["f"] & X.
\end{tikzcd}
 \]
\end{ex}

A map of operads $f:\mathcal{O} \to \mathcal{P}$ is a sequence of maps $f_n: \mathcal{O}(n) \to \mathcal{P}(n)$ such that $f_1(1) = 1 $  and commutes with the structure maps, i.e.  the diagram 
\[ 
\begin{tikzcd}
\mathcal{O}(n) \times (\mathcal{O}(j_1) \times \dots \times \mathcal{O}(j_n)) \rar[""] \dar["f_n \times (f_{j_1} \times \dots \times f_{j_n})"'] & \mathcal{O}(j_1 + \dots + j_n) \dar["f_{j_1 + \dots + j_n}"] \\
 \mathcal{P}(n) \times (\mathcal{P}(j_1) \times \dots \times \mathcal{P}(j_n)) \rar[""] & \mathcal{P}(j_1 + \dots + j_n)
\end{tikzcd}
\]
commutes for all $n, j_1, \dots , j_n \geq 0$.

\begin{defn} \label{defn:operadalgebra}
Let  $\mathcal{O}$ be a topological operad.  An $\mathcal{O}$-algebra structure on $X \in \Top$ is a map of operads
\[ \mu: \mathcal{O} \to \mathcal{E}(X). \]  
\end{defn}

 The above definition can easily be extended to define $\mathcal{O}$-algebra structure  for objects in $\Sp$. Note that $\Sp$ is closed symmetric monoidal, where the  suspension functor 
 \[ \sph[\ ]: \Top \to \Sp\]
 is monoidal and admits a right adjoint. Therefore, for any two spectrum $\mr{R}$ and $\mr{L}$ the collection of functions 
 \[ \Func(\mr{R},\mr{L})\]
 can be thought of as an object in $\Top$. As a result,  for any $\mr{R}\in \Sp$, we can define the endomorphism operad $\mathcal{E}(X)$, which is a topological operad with the $n$-th space 
 \[ \mathcal{E}(\mr{R})(n) = \Func(\mr{R}^{\sma n}, \mr{R} ).\]
 \begin{defn}
We say that $\mr{R} \in \Sp$  admits an $\mathcal{O}$-algebra structure if there exists a map of operads 
\[\mu: \mathcal{O} \to \mathcal{E}(\mr{R}).\]
\end{defn}
\subsection{Stasheff $\A_n$-operad}

In $1963$, Stasheff  \cite{Sta1} introduced a sequence of polytopes $\K^n$ which are now known as Stasheff polytopes. These polytopes are designed to describe a sequence of unital operads,  called the Stasheff $\A_n$-operads.  These operads can be used to describe a hierarchy of coherence for homotopy associative multiplications. 


The Stasheff polytope $\K^n$, as a topological space, is just homeomorphic to the disk $\mathbb{D}^{n-2}$, but encodes a rich cellular structure which parametrizes a homotopy coherent associative structure. The cells of $\K^n$ are indexed by the set of planar rooted trees with $n$ leaves. The polytopes $\K^1$ and $\K^2$ are just one-point spaces. The polytope $\K^3$ is the unit interval and its cellular structure is described in the picture below.
\input{./K3.tex}

A product structure on $X \in \Top$ is a map $\mu: X \times X \to X$. One can also think of the product structure as a map 
\[ \mu_2: \K^2 \cong * \to \Func(X \times X, X).\]
If this multiplication is homotopy associative then the homotopy can be thought of as a map 
\[ \mu_3:  \K^3 \cong [0,1]\to \Func(X \times X \times X, X) \] 
such that $\mu_3(0) = \mu_2 \circ (\mu_2 \times \mathbbm{1}_X)$ and $\mu_3(1) = \mu_2 \circ ( \mathbbm{1}_X \times \mu_2)$.

The polytope $\K^4$ is the  pentagon. Given a multiplication $\mu_2$, there are five different four-fold multiplications, producing five different maps from $X^{ \times 4}$ to $X$. These multiplications can be encoded by five different binary trees with four leaves. These trees label the five vertices of $\K^4$ (see Figure~\ref{fig:K4}). Moreover, if the multiplication is homotopy associative, i.e. $\mu_3$ exists, then one can construct homotopies between any two four-fold multiplications. These homotopies can be glued together to give a map 
\[ \partial \mu_4: S^1 \cong \partial K^4 \to \Func(X^{\times 4}, X).\]
The edges of $\K^4$ are denoted by the planar rooted trees with four leaves and exactly one internal vertex. The arrangement is such that the tree that represents an edge can be obtained by collapsing one of the branches of those trees that represent one of the adjacent vertices. 
If the map $\partial \mu_4$ is homotopic to a constant map, then we can use this homotopy to obtain a map 
\[ \mu_4: \K^4 \to \Func(X^{\times 4}, X).\]
Thus $\K^4$ parameterizes homotopy coherence of the associativity among the four-fold multiplications. 
\input{./K4.tex}

In general the cells of $\K^n$ are in one-to-one correspondence with the planar rooted trees with $n$ leaves. More specifically, the codimension $k$ cells are in bijection with the planar rooted trees with $n$ leaves and $k$ internal vertices.
\begin{notn} \label{notn:tree1}Let $\T_n$ be the collection of planar rooted trees with $n$ leaves and \[\T_* = \bigcup_{n}\T_n.\] For each $t \in \T_n$, let $\K\brk{t}$ denote the corresponding cell of $\K^n$.  
\end{notn}
\begin{defn} A \emph{corolla} is a planar rooted tree with no internal vertex.
\end{defn}

For every $n \in \mathbb{N}$,  $\T_n$ contains exactly one corolla. If $t \in \T_n$ is a corolla then $\K(t)$ is the Stasheff polytope $\K^{n}$. For any other tree $t \in \T_n$, we can obtain a set of corollas by breaking the tree off at each vertex. We call this set the \emph{corolla decomposition} of $t$ and denote it by $\Cl(t).$
  
\begin{ex} If $t$ is the tree 
\[ \begin{tikzpicture}
\tikzstyle{level 1}=[level distance=.08in]
		\tikzstyle{level 2}=[level distance=.08in, sibling distance=.08in]
		\tikzstyle{level 3}=[level distance=.08in, sibling distance=.08in]
		\tikzstyle{level 4}=[level distance=.08in, sibling distance=.08in]	
			\node[fill=none] at (0,0) {} [grow'=up]
				child {
					child {
						child
						child
						child			
					}
					child 
					child 
					child {
						child
						child { child 
							child
							child
						}
					}
				};
\end{tikzpicture}
\]
then $\Cl\brk{t}= \lbrace
\begin{tikzpicture}[baseline=1.5ex]
\tikzstyle{level 1}=[level distance=.08in]
		\tikzstyle{level 2}=[level distance=.08in, sibling distance=.08in]
		\tikzstyle{level 3}=[level distance=.08in, sibling distance=.08in]
		\tikzstyle{level 4}=[level distance=.08in, sibling distance=.08in]	
			\node[fill=none] at (0,0) {} [grow'=up]
				child {
					child 
					child 
					child 
					child 
				};
\end{tikzpicture}, 
\begin{tikzpicture}[baseline=1.5ex]
\tikzstyle{level 1}=[level distance=.08in]
		\tikzstyle{level 2}=[level distance=.08in, sibling distance=.08in]
		\tikzstyle{level 3}=[level distance=.08in, sibling distance=.08in]
		\tikzstyle{level 4}=[level distance=.08in, sibling distance=.08in]	
			\node[fill=none] at (0,0) {} [grow'=up]
				child {
					child 
					child 
					child 
				};
\end{tikzpicture},
\begin{tikzpicture}[baseline=1.5ex]
\tikzstyle{level 1}=[level distance=.08in]
		\tikzstyle{level 2}=[level distance=.08in, sibling distance=.08in]
		\tikzstyle{level 3}=[level distance=.08in, sibling distance=.08in]
		\tikzstyle{level 4}=[level distance=.08in, sibling distance=.08in]	
			\node[fill=none] at (0,0) {} [grow'=up]
				child {
					child 
					child 
				};
\end{tikzpicture},
\begin{tikzpicture}[baseline=1.5ex]
\tikzstyle{level 1}=[level distance=.08in]
		\tikzstyle{level 2}=[level distance=.08in, sibling distance=.08in]
		\tikzstyle{level 3}=[level distance=.08in, sibling distance=.08in]
		\tikzstyle{level 4}=[level distance=.08in, sibling distance=.08in]	
			\node[fill=none] at (0,0) {} [grow'=up]
				child {
					child 
					child 
					child
				};
\end{tikzpicture} \rbrace.$
\end{ex} 
 The cell $\K(t) \subset \K^n$ is homeomorphic to 
 \begin{equation} \label{corollaiso}
  \prod_{s \in \Cl(t)}\K(s) = \prod_{s \in \Cl(t) }\K^{l(s)}
  \end{equation}
where $l(s)$ denotes the number of leaves in the corolla $s$. This product is unique up to association.  There are various models for the Stasheff polytope $\K^n$. The first one is of course due to Stasheff \cite{Sta1}. Other prominent models include \cite{BV,CFZ, Lod, Tonk}. Define
\[ \K^k\langle n \rangle = \bigcup_{t \in \T_{k}\langle n \rangle }\K(t) \subset \K^{k} = \K^{k}\brk{\infty}. \]
\begin{figure}[h]
\centering
\begin{tikzpicture}[scale = .7]
\draw[fill = gray, opacity = .6, color = gray] (-1, 0) -- (1, 0) -- (1,2) -- (-1, 2) -- (-1,0);
\draw[color = gray] (-1, 0) -- (-2, 1) -- (0, 1.5) -- (2,1) -- (1, 0); 
\draw[color =gray] (-2,1) -- (-2,3) -- (-1.5, 4) -- (-1, 2);
\draw[color =gray] (2,1) -- (2,3) -- (1.5, 4) -- (1, 2);
\draw[color =gray] (-1.5, 4) -- (0,5) -- (1.5, 4);
\draw[fill = gray, opacity = .6, color = gray] (0,5) -- (1.5, 4) -- (2,3) -- (0,4) --(0,5);
\draw[fill = gray, opacity = .6, color = gray] (0,1.5) -- (-.5, 2.75) -- (-2,3) -- (-2,1) --(0,1.5);
\draw[color = gray] (-.5, 2.75) -- (0,4);
\end{tikzpicture}
\caption{The complex $\K^5 \langle 3 \rangle$}
\end{figure} 
\begin{rmk} \label{rem:boundary} When $k \leq n$, the set of trees with at most $n$ descendants from each vertex, include all the trees with $k$ leaves, i.e. $T_{k} = T_{k}\langle n\rangle.$  Therefore, \[ \K^{k}\langle n \rangle = \K^{k} \] 
for $k \leq n$.
\end{rmk}  
\begin{rmk} $T_{n+1}\langle n\rangle$ consists of all trees  $T_{n+1}$  except the corolla with $n+1$-leaves. Hence,  
\[ \K^{n+1}\langle n \rangle = \partial \K^{n+1}.\]
\end{rmk}

 \begin{defn} \label{defn:Anoperad}For $1 \leq n \leq \infty$,  the \emph{Stasheff $\A_n$-operad} is an operad formed from out of the sequence 
 \[ \A_{n}(k) =K^k \langle n \rangle, \]
 where $k \geq 0$.
 \end{defn}  
All but the degenerate structure maps of $\A_n$, at the cellular level, can be described in terms of concatenation of trees.  The cellular description of a  basic degenerate map
\begin{equation} \label{eqn:degeneracy}
s_i := \gamma(n; \underbrace{1, \dots, 1}_{i-1}, 0 , 1 \dots, 1): \K^k\brk{n} \to \K^{k-1}\brk{n} 
\end{equation}
correspond to deleting the $i$-th leaf. The point-set level description of these maps are well-documented, for example see \cite{Sta1}.

\subsection{Obstruction theory for $\A_n$-structures}

 Let $n$ be a finite positive integer thorught this subsection. An $\A_{n}$-algebra structure on $X \in \Top$ equivalent to a collection of  
\[ \mu_k\brk{n}: \K^k\brk{n} \times X^{\times n} \to X \]
which satisfies the usual compatibility criterions. Since $\K^k\brk{n} = \K^k$ for $k \leq n$, we will denote abbreviate $\mu_k\brk{n}$ to $\mu_k$ for the remainder of the paper. Also note that when  $k>n$ each cell of $\K^k\brk{n}$ is isomorphic to a finite product of $\K^i$ for $1 \leq i \leq n$. As a result the finite collection of maps
\[ \{ \mu_0, \dots, \mu_n \}\]
determines $\mu_k\brk{n}$ for $k>n$. This leads to an obstruction theory.

Suppose, $X \in \Top$ admits an $\A_{n-1}$-structure. If we want to extend this structure to an $\A_n$-structure, then we simply need to  define a map 
\[ \mu_n: \K^n \times X^{\times n } \to X\]
which is compatible with $\{ \mu_0, \dots, \mu_{n-1} \} $ determined by the $\A_{n-1}$-structure. The compatibility criteria emerges from that fact that the $\A_{n-1}$-structure already determines the map $\mu_n$ restricted to certain subspaces of $\K^n \times X^{\times n}$. 

The fact that $\K^{n}\brk{n-1} = \partial \K^n$ (see \Cref{rem:boundary}) implies $\mu_n$ must restrict to 
\begin{equation} \label{eqn:res1}
 \begin{tikzcd}
 \mu_{n-1} \brk{n}: \partial \K^n \times X^{\times n } \rar[] & \K^n \times X^{\times n } \rar["\mu_n"] & X 
 \end{tikzcd}
 \end{equation}
Another collection of restrictions come from the fact that the operad $\A_n$, by virtue of being a unital operad, admits degeneracy maps (see \Cref{eqn:degeneracy}). Let 
\[
\begin{tikzcd}
 \delta_i: X^{\times (n-1)} \arrow{rrrr}{1^{\times (i-1)} \times \mu_0 \times 1^{\times(n-i-1)}} &&&& X^{\times n}
 \end{tikzcd} \]
 where $\mu_0: \ast \to X$ denote the unit map. Then $\mu_n$ must make the diagram 
 \begin{equation} \label{eqn:res2}
 \begin{tikzcd}
 \K^n \times X^{\times (n-1) } \dar["s_i \times 1"'] \rar["1 \times \delta_i"] & \K^n \times X^{\times n} \dar["\mu_n"]  \\
\K^{n-1}\times X^{\times (n-1) }  \rar["\mu_{n-1}"] & X
 \end{tikzcd}
  \end{equation}
commute. If we define 
\[ X^{\times n}_{[n-1]}:= \{(x_1, \dots, x_n): \text{$x_i = \ast$ for some $i$}  \} \subset X^{n}, \]
where $\ast$ is the image of $\mu_0$, then $\A_{n-1}$-structure on $X$ determines a map
\[ \mu_{[n-1]}: \K^n \times X^{\times n}_{[n-1]} \to X  \]
and  \Cref{eqn:res2} implies $\mu_n$ must restrict to $\mu_{[n-1]}$ on $\K^n \times X^{\times n}_{[n-1]}$. The map $\mu_{[n-1]}$ 

Let $\Phi_n(X)$ denote the pushout 
\begin{equation} \label{eqn:Phi_n}
\begin{tikzcd} 
\partial \K^n \times X^{\times n}_{[n-1]} \rar \dar & \K^n \times X^{\times n}_{[n-1]} \dar \\
\partial \K^n \times X^{\times n} \rar & \Phi_n(X). 
\end{tikzcd}
\end{equation}
 The $\A_{n-1}$-structure on $X$ induces a map 
\[ \Phi_n(\mu): \Phi_n(X) \to X. \]
There is a natural map $\Phi_n(X) \to \K^n \times X^{\times n}$ and by \Cref{eqn:res1} and \Cref{eqn:res2}, any map $ \mu_n: \K^n \times X^{\times n} \to X $ which restricts to $\Phi_n(\mu)$ extends the $\A_{n-1}$-structure on $X$ to an $\A_n$-structure. Suppose there is an object $\sigma_n(X)$ such that 
\[ \sigma_n(X) \to \Phi_n(X) \to \K^n \times X^{\times n}\]
is a cofiber sequence, then we get the following theorem. 
\begin{thm}[Stasheff] An $\A_{n-1}$-algebra structure on $X \in \Top$ can be extended to an $\A_n$-structure if and only if the composite 
\[  
\begin{tikzcd}
\sigma_n(X) \rar & \Phi_n(X) \rar["\Phi_n(\mu)"]& X 
\end{tikzcd} \]
is homotopic to the trivial map.
\end{thm}

A similar obstruction theory holds for $\A_n$-structures on objects of $\Sp$. Suppose $\R\in \Sp$ admits an $\A_{n-1}$-algebra structure, i.e. a collection of maps 
\[ \mu_i: \K^i_+ \sma \R^{\sma i} \to \R  \]
for $0 \leq i \leq n-1$. Assuming  the unit map 
\[ \mu_0: \sph \to \R \] 
to be cofibration, one can construct the object $\R^{\sma n}_{[n-1]}$ as a colimit. Let $[n]$ denote the category of ordered subsets of the order set $\{ 1 < \dots < n \}$. Let $\mathcal{C}(n,n-1)$ be the full subcategory consisting of subsets of at most $n-1$ elements. Let 
\[ \mathcal{F}^\R_{n,n-1}: \mathcal{C}(n,n-1) \to \Sp \]
denote the functor which sends a subset with $i$ element to $\R^{\sma i}$ and uses the unit map $\mu_0$ to establish the necessary maps between $\R^{\sma i}$ and $\R^{\sma j}$. Define 
\[ \R^{\sma n}_{[n-1]} := \colim_{\mathcal{C}(n,n-1)} \mathcal{F}^{X}_{n,n-1}. \]
One can construct a spectrum  $\Phi_n(\R)$ by replacing $\K^n$ by $\K^n_+$, the cartesian product $\times$ by the smash product $\sma$ and $X$ by $\R$ in the \Cref{eqn:Phi_n}.  Since cofiber sequences are fiber sequences in $\Sp$, the object $\sigma_n(\R)$ always exists, it is the fiber of the map 
\[ 
\begin{tikzcd} 
\Phi_n(\mu): \Phi_n(\R) \rar & \K^n_+ \sma \R^{\sma n}.
\end{tikzcd}
\] 
\begin{thm}[Stasheff] \label{thm:Anspec} An $\A_{n-1}$-algebra structure on $\R  \in \Sp$ can be extended to an $\A_n$-structure if and only if the composite 
\[  
\begin{tikzcd}
\sigma_n(\R) \rar & \Phi_n(\R) \rar["\Phi_n(\mu)"]& \R 
\end{tikzcd} \]
is homotopic to the trivial map.
\end{thm}
\subsection{ A conjecture of Mahowald}
\Cref{thm:Anspec} has some immediate application in the context of two-cell complexes. Let $\tau \in \pi_{k-1}(\sph)$ and let $\mr{C}(\tau)$ denote the cofiber of $\tau$. Since the discussion takes place in $\Sp$ (and not in the homotopy category $h \hspace{-1.7pt}\Sp$), it is important to give $\mr{C}(\tau)$ an explicit pointset model. In our category of spectra, the sphere spectrum $\sph$ is fibrant but not cofibrant, therefore $\tau$ may not be realized as a map from $\Sigma^{k-1}\sph$ to $\sph$. However, $\tau$ can be realized as a map 
\[ \tau: \Sigma^{k-1}\csph \to \sph,\]
where $\csph$ is the cofibrant relacement of $\sph$. We specifically choose $\csph$ to be the cofibrant replacement of \cite[Chapter~$2$]{EKMM}. Moreover, we choose and fix non-canonical and non-coherent isomorphisms 
\[\underbrace{(\csph \sma \dots \sma \csph)}_{\text{$n$ times}}\iso \csph \]
 for all $n>0$. What we gain is the isomorphism 
\begin{equation} \label{eqn:specialiso}
 (\Sigma^k \csph)^{\sma n} \iso \Sigma^{kn}\csph 
\end{equation}
 for all $k \geq 0$ and $n>0$.
\begin{notn} For efficiency of notations, we will denote $\Sigma^{n}\csph$ by $\csph^n$ and  $C\Sigma^{n}\csph$ (the cone on $ \Sigma^{n}\csph$) by $\disk^{n+1}$. 
\end{notn}
We choose the pointset model of $\mr{C}(\tau)$ to be the pushout in the diagram 
\begin{equation} \label{eqn:pointset} 
\xymatrix{
\csph^{k-1} \ar[r]^{\tau} \ar[d] & \sph \ar[d] \\
\disk^k \ar[r] & \mr{C}( \tau).
} 
\end{equation} 
The `inclusion map' from $\sph \to \mr{C}(\tau)$ serves as the unit map. With this model of $\mr{C}(\tau)$ and Equation~\ref{eqn:specialiso}, it can be easily seen that $\K^n_+ \sma \mr{C} (\tau)^{\sma n}$ is related to $\Phi_n(\mr{C}(\tau))$ via the pushout diagram  
\begin{equation} \label{eqn:attach1}
\xymatrix{
\csph^{n(k+1)-3} \ar[r]^-{\iota} \ar[d]& \Phi_n(\mr{C}(\tau)) \ar[d]\\
\disk^{n(k+1)-2} \ar[r]_-{\tilde{\iota}} & \K^n_+ \sma (\mr{C}(\tau))^{\sma n}.
}
\end{equation}
This discussion shows that $\sigma_n(\mr{C}(\tau)) \simeq \Sigma^{n(k+1)-3}\sph$. Applying \Cref{thm:Anspec}, we get:
\begin{cor} \label{cor:Antwocell}Let $\mr{C}(\tau)$ denote the cofiber of $\tau \in \pi_{k-1}(\sph)$. The obstruction to extending  an $\A_{n-1}$-algebra structure on $\mr{C}(\tau)$  to an $\A_n$-structure is a class in $\pi_{n(k+1) -3}(\mr{C}(\tau))$. 
\end{cor}
Here are some straightforward applications of \Cref{cor:Antwocell}
\begin{ex} \label{ex:oddmoore} When $p$ is odd, $\pi_i(\M{p}{i}) = 0 $ for $1 \leq i \leq 2p-4$. It follows that $\M{p}{i}$ admits $\A_{p-1}$-structure for all $i$. 
\end{ex}
\begin{ex} \label{ex:moore2} Oka \cite{Oka} showed that $\M{2}{i}$ admits an $\A_3$-structure for $i \geq 2$. Since $\pi_5(\M{2}{i}) = 0$, it follows that $\M{2}{i}$ admits $\A_4$-structure for $i \geq 4$. This gives us the best possible answer for $\M{2}{2}$ and $\M{2}{3}$ and is a part of the data in \Cref{table2}.
\end{ex}
One do not expect $\M{p}{i}$ to admit an $\A_\infty$-structure, perhaps because of the following general conjecture for two-cell complexes which Mark Mahowald communicated to the author during a private conversation.
\begin{conj}[Mahowald] \label{conj:mah} $\mr{C}(\tau)$ admits an $\A_\infty$-structure if and only if $\tau \simeq \ast$. 
\end{conj}

\subsection{Homotopy $\A_n$-maps }
An $\A_n$-map $f: X \to Y$ between two objects in $\Top$ with $\A_n$-structure, is a map for which the diagram 
\[ 
\begin{tikzcd}
\K^{i} \times X^{\times n} \rar["\mu_i^X"] \dar["1 \times f^{\times i}"'] & X \dar["f"] \\
\K^{i} \times Y^{\times n} \rar["\mu_i^Y"']  & Y 
\end{tikzcd}
\]
commutes for all $0 \leq i \leq n $. A homotopy $\A_n$-map is a map where we only require the above diagrams to commute up to homotopy, however our choices of homotopy must satisfy some coherence conditions. These coherences are best described in using a system of polytopes $\{ \mr{J}^n: n \geq 1 \} $ called \emph{multiplihedra}. 

 $\mr{J}^1$ is simply a point. $\mr{J}^2$ is the unit interval $[ 0, 1]$.  A homotopy $\A_2$-map consists of $f_1: X \to Y$ which preserves the unit, a homotopy 
\[ f_2: \mr{J}^2 \times X^{\times 2} \to Y \]
 which parametrizes a homotopy 
\[  \mu_2^Y\circ (f \sma f) \simeq f \circ \mu_2^X.\]
The two natural restriction of  $f_2$ to $\mr{J}^2 \times X$ via the unit map of $X$ must be the constant homotopy at $f_1$. $\mr{J}^3$ is a hexagon. To extend a homotopy $\A_2$-map to a homotopy $\A_3$-map, one must provide a map 
\[ f_3: \mr{J}^3 \to \Func(X^{\times 3}, Y)\]
which is compatible with predefined maps $f_1$ and $f_2$.
 \begin{figure}[h]  \label{fig:J3}
\centering
\begin{tikzpicture}
\draw (-1, 0) -- (1,0) -- (2.5, 2) -- (1,4) -- (-1, 4) -- (-2.5, 2) -- (-1, 0);
\node[fill=none] at (-2.3,0) {{\tiny $(f(x_1)f(x_2))f(x_3) - $}};
\node[fill=none] at (2.3,0) {{\tiny $ - f(x_1)(f(x_2)f(x_3))$}};
\node[fill=none] at (0,-.3) {{\tiny $ f(x_1)f(x_2)f(x_3)$}};
\node[fill=none] at (-2.9,1) {{\tiny $f_2(x_1,x_2)f(x_3)$}};
\node[fill=none] at (-3.5,2) {{\tiny $f(x_1x_2)f(x_3) - $}};
\node[fill=none] at (-2.9,3.1) {{\tiny $f_2(x_1x_2,x_3)$}};
\node[fill=none] at (-2,4) {{\tiny $f((x_1x_2)x_3) - $}};
\node[fill=none] at (0,4.3) {{\tiny $f(x_1x_2x_3)$}};
\node[fill=none] at (2,4) {{\tiny $-f(x_1(x_2x_3))$}};
\node[fill=none] at (2.6,3.1) {{\tiny $f_2(x_1, x_2x_3)$}};
\node[fill=none] at (3.4,2) {{\tiny $ - f(x_1)f(x_2x_3)$}};
\node[fill=none] at (2.8,1) {{\tiny $f(x_1)f_2( x_2, x_3)$}};
\end{tikzpicture}
\caption{The polytope $\mr{J}^3$}
\end{figure} 
In general, the compatibility criteria arises from the fact that the maps $\{f_r: 1 \leq r \leq n-1\}$ along with the maps \[ \{\mu_r^X,\mu_r^Y : 0 \leq r \leq n-1 \} \] determine the map $\partial f_n$ defined on $\partial \mr{J}^n$. On the boundary of $\mr{J}^n$ there are two disjoint cells homeomorphic to $\K^n$, one of which supports the map $ (f_1^{\times n}) \circ \mu_n^Y$ and the other supports $f_{1} \circ \mu_n^X$. Therefore, the map $f_n$ in some sense is a `homotopy with additional coherence conditions' between these two maps. 

Multiplihedra and their connection to homotopy $A_n$-maps were first considered by Stasheff~\cite{Sta1}. Boardman and Vogt~\cite{BV} expressed the full combinatorial descriptions of these multiplihedra using the language of colored operads and metric trees. \cite{DF, F, IM} are  among prominent articles in the literature containing detailed descriptions of multiplihedra. 

\subsection{An obstruction theory for homotopy $\A_n$-maps}

 Stasheff developed an obstruction theory for homotopy $\A_n$-maps in the category of $\Top$, by constructing what is called a \emph{truncated bar complex} for $\A_n$-algebras. It was using these techniques, Stasheff \cite{StasheffPoV}[Theorem~$7.4$] proved that $\mr{S}^7$ does not admit a homotopy associative multiplication. We partially adopt \cite{VKT2}[$\mathsection$ 4] in our description of the Stasheff's bar construction.
 
 As a warm up, we briefly recall the construction of a bar complex for a strictly associative monoid in $\Top$. A strictly associative monoid is an object $\H \in \Top$ with a unit map $\iota: \ast \to \H$ and a multiplication map $$\mu: \H \times \H \to \H$$ which is compatible with the unit map, i.e. $\mu\circ (\iota \times 1_\H) = 1_\H = \mu \circ (1_\H \times \iota)$, and is strictly associative, i.e. $\mu \circ (\mu \times 1_\H) = \mu \circ (1_\H \times \mu)$. A left $\H$-module $\mr{M}$ is an object in $\Top$ with a map 
\[ m: \H \times \mr{M} \longrightarrow \mr{M} \]
which satisfies the usual conditions. Similarly, a right $\H$-module $\mr{M}$ is an object in $\Top$ with a map 
\[ n: \mr{N} \times \H \longrightarrow \mr{N} \]
satisfying the usual conditions. 
Let $\Delta$ be the category of finite, nonempty, totally ordered sets with order preserving maps as morphisms. Let $\mr{sk}(\Delta)$ denote the skeleton category of $\Delta$. Objects of $\mr{sk}(\Delta)$ are finite ordinals. We denote the ordinal $n+1$ by $[n]$ or $\lbrace 0< \dots < n \rbrace$. Given a strictly associative monoid $\H$, a left $\H$-module $\mr{M}$ and right $\H$-module $\mr{N}$, we can define a functor 
\[ \mathcal{B}(\mr{M},\H,\mr{N}): \mr{sk}(\Delta)^{op} \longrightarrow \Top \] 
which sends $[n] \mapsto \mr{M} \times \H^{n} \times \mr{N}$. On the other hand, we have a functor 
\[ | \Delta|: \mr{sk}(\Delta) \to \Top\]
such that $[n] \mapsto \Delta_n$, where $\Delta_n$ is the geometric  $n$-simplex. The two sided bar-complex $\mr{B}(\mr{M},\H,\mr{N})$ is the coend of the functor $ \mathcal{B}(\mr{M},\H,\mr{N}) \times |\Delta |$ or equivalently the quotient space 
\[ \mr{B(M,H,N)} = \coprod_{n \geq 0} \Delta_n \times (\mr{M} \times \H^{n} \times \mr{N})/ \sim\]
where $\sim$ is the usual identification expressed in terms of face and degeneracy maps. We define 
\[ \mr{BH} = \mr{B}(\ast, \H,\ast)\]
as the bar complex of $\H$. 

For an $\A_{n}$-algebra $\H$ in $\Top$, a \emph{right $\A_k$ $\H$-module} $M$, is an object in $\Top$ with maps 
\[ f_r:\K^{r+1} \times \mr{M} \times \mr{H}^{\times r} \to \mr{M}\]
for $1\leq  r \leq k$, satisfying the usual compatibility conditions with the higher order multiplication of $H$. Similarly, a \emph{left $\A_k$ $H$-module} $N$ is an object in $\Top$  with maps 
\[ g_r:\K^{r+1} \times \H^{\times r} \times \mr{N} \to \mr{M}\]
for $0 \leq  r \leq k$, satisfying similar compatibility conditions. The two-sided bar construction described above does not make sense because $\mathcal{B}(M,H,N)$ fails to be a functor when the multiplication is not strictly associative. However, this issue can be resolved. Roughly speaking, the idea is to inflate the morphism classes between two objects in $\mr{sk}(\Delta)^{\mr{op}}$ to accommodate $\A_n$-structures. More precisely, we enrich the category $\mr{sk}(\Delta)^{\mr{op}}$ over $\Top$ using the operad $\A_{\infty}$ by setting the morphism class between $[l]$ and $[k]$ as the topological space 
\[ \bigsqcup_{f: [l] \to [k]} \K^{[f]},\]
where $\K^{[f]} = \prod_{i \in [k]} \K^{f^{-1}(i)}$. Denote the resultant category by $\Delta^{\mr{op}}_{\infty}$.

 For an $\A_{\infty}$-algebra $\H$, it is straightforward to verify that 
\[ \mathcal{B}_{\infty}(\mr{M},\H,\mr{N}): \Delta^{\mr{op}}_{\infty} \to \Top \]
sending $[n] \mapsto (\mr{M} \times \H^{n} \times \mr{N})$ is indeed a functor. For brevity, let us denote $(\Delta^{\mr{op}}_{\infty})^{\mr{op}}$ by $\Delta_{\infty}$. There is also a canonical functor
\[ |\K|: \Delta_{\infty} \to \Top\]
where $[n] \to \K^{n+2}$. 

\begin{defn} \label{defn:barfinity}
For an $\A_{\infty}$-algebra $\H$, a right $\A_{\infty}$ $\H$-module $\mr{M}$ and a left $\A_{\infty}$ $\H$-module $\mr{N}$, the two-sided bar complex $\mr{B(M,H,N)}$ is the coend 
\[ \mr{B(M,H,N)} = \int^{\Delta_{\infty}} \mathcal{B}_{\infty}(\mr{M,H,N}) \times |\K|.\] 
\end{defn}
\begin{defn} For an $\A_{\infty}$-algebra $\H$, the bar complex of $\H$ is the topological space 
\[ \mr{BH} = \mr{B}(\ast, \H, \ast)\]
\end{defn}
\begin{rmk}We intentionally used the same notation for the two sided bar-complex $\mr{B(M,H,N)}$, when $\H$ is a strictly associative monoid and when $\H$ is a $\A_{\infty}$-algebra. This is because a strictly associative monoid in $\Top$ is automatically an $\A_{\infty}$-algebra and the two different bar constructions yield the same space up to homotopy.   
\end{rmk}
 For an $\A_n$-algebra $\H$ where $n< \infty$, one can only construct a truncated version of the bar complex, called the $n$-truncated bar complex, simply by replacing $\Delta_{\infty}$
 by $\Delta_{\leq n}$, the full subcategory of $\Delta_{\infty}$ with objects $[k]$ for $0 \leq k \leq n$, in the discussion above.
 
 For an $\A_{n}$-algebra $\H$, a right $\A_{n}$ $\H$-module $\mr{M}$ and a left $\A_{n}$ $\H$-module $\mr{N}$ let 
\[ \mathcal{B}_n(\mr{M},\H,\mr{N}): \Delta_n^{op} \to \Top\]
be the functor that sends $[k] \to \mr{M} \times \H^{k} \times \mr{N}$. We also have a functor 
\[ |\K^n|: \Delta_{\leq n} \to \Top\]
such that $[k] \mapsto \K^{k+2}$.  
\begin{defn}
For an $\A_{n}$-algebra $\H$, a right $\A_{n}$ $\H$-module $\mr{M}$ and a left $\A_{n}$ $\H$-module $\mr{N}$, the two-sided bar complex $\mr{B(M,H,N)}$ is the coend 
\[ \mr{B}_n(\mr{M},\H,\mr{N}) = \int^{\Delta_{\leq n}} \mathcal{B}_{n}(\mr{M},\H,\mr{N}) \times |\K^n|.\] 
\end{defn}
The space $\mr{B}_n(\mr{M, H, N})$ is a quotient of the space 
\[ \coprod_{0\leq k \leq n}\K^{k+2} \times \mr{M} \times \H^k \times \mr{N}.\]
\begin{defn} For an $\A_{n}$-algebra $\H$ the bar complex of $\H$ is the topological space 
\[ \mr{B}_n\H = \mr{B}_n(\ast, \H, \ast)\]
\end{defn}
\begin{ex} $\mr{S}^1$ admits a  strict associative multiplication when viewed as units in $\mathbb{C}$. The $n$-truncated bar complex $\mr{B}_n\mr{S}^1$ is homotopy equivalent to the projective space $\CP^n$.
\end{ex}
\begin{ex} $\mr{S}^3$ admits a  strict associative multiplication when viewed as units in quaternion $\mathbb{H}$. The $n$-truncated bar complex $\mr{B}_n\mr{S}^3$ is homotopy equivalent to the projective space $\mathbb{HP}^n$.
\end{ex}
\begin{ex}   $\mr{S}^7$ admits a multiplication when viewed as units in octonions $\mathbb{O}$. However, $\mr{S}^7$ cannot support an homotopy associative multiplication. If it does, then we can construct $\mr{B}_3\mr{S}^7 \cong \mathbb{OP}^3$, and its cohomology ring is the truncated polynomial ring \[ \H^*(\mathbb{OP}^3; \mathbb{Z}/3 )\cong \mathbb{Z}/3[u_8]/(u_8^4),\] where $u_8$ is a generator in degree 8.  By axioms of Steenrod operation $\mathcal{P}^4(u_{8})$ must equal $u_8^3$. However, Adem relation $\mathcal{P}^4 = \mathcal{P}^1 \mathcal{P}^3$ and  the fact that  $$\mathcal{P}^3(u_8) \in \H^{20}(\mathbb{OP}^3; \mathbb{Z}/3 ) = 0,$$  imply $\mathcal{P}^4(u_8)$ must equal zero, which is a contradiction. 
\end{ex} 

In this paper we make use of the following result of Stasheff, and a rigidification theorem of Boardman and Vogt thereafter.

\begin{thm}[Stasheff] \label{thm:barasso} A map $f: X \to Y$, where $X$ and $Y$ are  strictly associative monoids in $\Top$, extends to a unital homotopy $A_n$-map if and only if the map $\Sigma f:\Sigma X \to \Sigma Y$ extends to a map \[ \mr{B}_n f : \mr{B}_n X \to \mr{B}_n Y \]
where $\mr{B}_n X$ and $\mr{B}_n Y$ are $n$-truncated bar complexes of $X$ and $Y$ respectively.
\end{thm}
 
Work of Boardman and Vogt \cite[Chapter~4]{BV} provides a technique to replace a homotopy $\A_n$-map (unital or nonunital) by an $\A_n$-map in the homotopy category of $\A_n$-algebras. Given a topological operad $\mathcal{O}$, they make two important constructions, 
\begin{itemize} 
\item an endofunctor $\mr{W}$ from the category of topological operads to itself such that there is a natural map of operads
 \[ w: \mr{W} \mathcal{O} \to \mathcal{O} \]
which is a weak equivalence, and, 
\item a functor $\mr{U}:\Top[\mathcal{O}] \to \Top[\mathcal{O}]$, where $\Top[\mathcal{O}]$ is the category of $\mathcal{O}$-algebras in $\Top$, such that there is a natural map 
\[ u: X \to \mr{U}X \]
which is a weak equivalence and can be extended to a `homotopy $\mathcal{O}$-map' (see Remark~\ref{rem:translate}). This map is called the 
\emph{universal homotopy $\mathcal{O}$-map}.
 \end{itemize} 
These constructions conspire to give us the following theorem, which is essentially a special case of  \cite[Theorem~4.23(c)]{BV}. 
\begin{thm} \label{thm:Anfix}Let  $X$ and $Y$ be $\A_n$-algebras and  $u:X \to\mr{U}X$ be the universal  homotopy $\A_n$-map. Then
 for any  homotopy $\A_n$-map  $f:X \to Y$ has a unique factorization 
\[ 
\xymatrix{
X \ar[r]^f \ar[d]_{u}^{\simeq}& Y \ar@{=}[d] \\
\mr{U}X \ar[r]^{h} & Y
}
\]
such that $h$ is a $\mr{W}\A_n$-map.
\end{thm} 
The terminologies used in \cite{BV} are significantly different from the ones used in this paper. This can be a potential source of confusion. Hence, in the following remark, we provide a dictionary between the language in \cite{BV} and the language used in this paper. 
\begin{rmk} \label{rem:translate} In \cite{BV} authors study $K$-colored topological algebraic theories  $\mathcal{B}$ for a finite set $K$ called colors \cite[Definition~2.3]{BV}. For any $K$ colored theory $\mathcal{B}$, they define $\mathcal{B}$-spaces, $\mathcal{B}$-homomorphisms, homotopy $\mathcal{B}$-homomorphisms (which they abbreviated to $\mathcal{B}$-maps \cite[Definition~4.1]{BV}) and homotopy homogeneous $\mathcal{B}$-homomorphisms (which they abbreviated to $h\mathcal{B}$-maps \cite[Definition~4.2]{BV}).   An operad $\mathcal{B}$ in our language is a topological algebraic theory with one color (i.e. $K = \lbrace * \rbrace$) in their language,  a $\mathcal{B}$-algebra in our language is a $\mathcal{B}$-space in their language and a $\mathcal{B}$-map in our language is a $\mathcal{B}$-homomorphism in their language. If we were to define a homotopy $\mathcal{B}$-map between two $\mathcal{B}$-algebras in our language, it would have been equivalent to the definition of  homotopy homogeneous $\mathcal{B}$-map, i.e. an $h\mathcal{B}$-map in their language. Specifically, when $\mathcal{B} = \A_n$, `homotopy $\A_n$-map'  is an $h\mathcal{\A}_n$-map in the language used in \cite{BV}. This can be deduced from the discussions in \cite[Chapter~1.3]{BV} and \cite[Example~2.56]{BV}.
 \end{rmk}

%% file: K3.tex
\begin{figure}[h] \label{fig:K3}
\centering
\begin{tikzpicture}
\node[left] at (0,0) {$\bullet$};
\node[right] at (3,0) {$\bullet$};
\draw (0,0) -- (3,0);
\tikzstyle{level 1}=[level distance=.08in]
		\tikzstyle{level 2}=[level distance=.08in, sibling distance=.08in]
		\tikzstyle{level 3}=[level distance=.08in, sibling distance=.06in]
		\tikzstyle{level 4}=[level distance=.08in, sibling distance=.08in]	
			\node[fill=none] at (-.2,.2) {} [grow'=up]
				child {
					child {
						child
						child [fill=none] {edge from parent[draw=none]}
						child
					}
					child [fill=none] {edge from parent[draw=none]}
					child 
				};
\tikzstyle{level 1}=[level distance=.08in]
		\tikzstyle{level 2}=[level distance=.08in, sibling distance=.08in]
		\tikzstyle{level 3}=[level distance=.08in, sibling distance=.06in]
		\tikzstyle{level 4}=[level distance=.08in, sibling distance=.08in]	
			\node[fill=none] at (3.2,.2) {} [grow'=up]
				child {
					child 
					child [fill=none] {edge from parent[draw=none]}
					child {
						child
						child [fill=none] {edge from parent[draw=none]}
						child
					}
				};
\tikzstyle{level 1}=[level distance=.16in]
		\tikzstyle{level 2}=[level distance=.16in, sibling distance=.16in]
		\tikzstyle{level 3}=[level distance=.16in, sibling distance=.12in]
		\tikzstyle{level 4}=[level distance=.16in, sibling distance=.16in]	
			\node[fill=none] at (1.5,.2) {} [grow'=up]
				child {
					child 
					child 
					child 
				};

\end{tikzpicture}
\caption{Cellular structure of $\K^3$ expressed in terms of trees}
\end{figure}

%% file: K4.tex
\begin{figure}[h] \label{fig:K4}
\centering
\begin{tikzpicture}
\coordinate (x1) at (0, 2);
\coordinate (x2) at (-4, 4);
\coordinate (x3) at (-1.5, 5);
\coordinate (x4) at (1.5, 5);
\coordinate (x5) at (4, 4);
\draw[fill, color= gray, fill opacity=.4] (x1) -- (x2) -- (x3) -- (x4) -- (x5) -- cycle;
\draw (x1) -- (x2) -- (x3) -- (x4) -- (x5) -- cycle;
 \tikzstyle{level 1}=[level distance=.08in]
		\tikzstyle{level 2}=[level distance=.08in, sibling distance=.08in]
		\tikzstyle{level 3}=[level distance=.08in, sibling distance=.06in]
		\tikzstyle{level 4}=[level distance=.08in, sibling distance=.08in]	
			\node[fill=none] at (0,2.2) {} [grow'=up]
				child {
					child {
						child
						child [fill=none] {edge from parent[draw=none]}
						child
					}
					child [fill=none] {edge from parent[draw=none]}
					child {
						child
						child [fill=none] {edge from parent[draw=none]}
						child
					}
				};

		\tikzstyle{level 1}=[level distance=.06in]
		\tikzstyle{level 2}=[level distance=.06in, sibling distance=.06in]
		\tikzstyle{level 3}=[level distance=.06in, sibling distance=.06in]
		\tikzstyle{level 4}=[level distance=.06in, sibling distance=.06in]	
		\node[fill=none] at (-4,4.1) {} [grow'=up]
				child {
					child {
						child {
							child
							child [fill=none] {edge from parent[draw=none]}
							child
						}
						child [fill=none] {edge from parent[draw=none]}
						child
					}
					child [fill=none] {edge from parent[draw=none]}
					child 
				};
		\tikzstyle{level 1}=[level distance=.04in]
		\tikzstyle{level 2}=[level distance=.04in, sibling distance=.04in]
		\tikzstyle{level 3}=[level distance=.04in, sibling distance=.04in]
		\tikzstyle{level 4}=[level distance=.04in, sibling distance=.04in]	
		\node[fill=none] at (-1.5,5.1) {} [grow'=up]
				child {
					child {
						child 
						child [fill=none] {edge from parent[draw=none]}
						child{
							child
							child [fill=none] {edge from parent[draw=none]}
							child
						}
					}
					child [fill=none] {edge from parent[draw=none]}
					child 
				};	
		\tikzstyle{level 1}=[level distance=.04in]
		\tikzstyle{level 2}=[level distance=.04in, sibling distance=.04in]
		\tikzstyle{level 3}=[level distance=.04in, sibling distance=.04in]
		\tikzstyle{level 4}=[level distance=.04in, sibling distance=.04in]	
		\node[fill=none] at (1.5,5.1) {} [grow'=up]
				child {
					child 
					child [fill=none] {edge from parent[draw=none]}
					child {
						child {
							child
							child [fill=none] {edge from parent[draw=none]}
							child
						} 
						child [fill=none] {edge from parent[draw=none]}
						child
					}
				};
		\tikzstyle{level 1}=[level distance=.06in]
		\tikzstyle{level 2}=[level distance=.06in, sibling distance=.06in]
		\tikzstyle{level 3}=[level distance=.06in, sibling distance=.06in]
		\tikzstyle{level 4}=[level distance=.06in, sibling distance=.06in]	
		\node[fill=none] at (4,4.1) {} [grow'=up]
				child {
					child 
					child [fill=none] {edge from parent[draw=none]}
					child {
						child 
						child [fill=none] {edge from parent[draw=none]}
						child{
							child
							child [fill=none] {edge from parent[draw=none]}
							child
						}
					}
				};	
	\tikzstyle{level 1}=[level distance=.08in, color = blue]
		\tikzstyle{level 2}=[level distance=.08in, sibling distance=.08in, color = blue]
		\tikzstyle{level 3}=[level distance=.08in, sibling distance=.08in, color = blue]
		\tikzstyle{level 4}=[level distance=.08in, sibling distance=.08in, color = blue]	
			\node[fill=none] at (-2,3) {} [grow'=up]
				child {
					child {
						child
						child [fill=none] {edge from parent[draw=none]}
						child
					}
					child 
					child 
				};
		\tikzstyle{level 1}=[level distance=.06in, color = blue]
		\tikzstyle{level 2}=[level distance=.06in, sibling distance=.06in, color = blue]
		\tikzstyle{level 3}=[level distance=.06in, sibling distance=.06in, color = blue]
		\tikzstyle{level 4}=[level distance=.06in, sibling distance=.06in, color = blue]	
			\node[fill=none] at (-2.5,4.6) {} [grow'=up]
				child {
					child {
						child
						child 
						child
					}
					child [fill=none] {edge from parent[draw=none]}
					child 
				};
		\tikzstyle{level 1}=[level distance=.05in, color = blue]
		\tikzstyle{level 2}=[level distance=.06in, sibling distance=.06in, color = blue]
		\tikzstyle{level 3}=[level distance=.04in, sibling distance=.04in, color = blue]
		\tikzstyle{level 4}=[level distance=.06in, sibling distance=.06in, color = blue]	
			\node[fill=none] at (0,5.05) {} [grow'=up]
				child {
					child 
					child {
						child
						child [fill=none] {edge from parent[draw=none]}
						child
					}
					child 
				};
		\tikzstyle{level 1}=[level distance=.06in]
		\tikzstyle{level 2}=[level distance=.06in, sibling distance=.06in, color = blue]
		\tikzstyle{level 3}=[level distance=.06in, sibling distance=.06in, color = blue]
		\tikzstyle{level 4}=[level distance=.06in, sibling distance=.06in, color = blue]	
			\node[fill=none] at (2.5,4.6) {} [grow'=up]
				child {
					child 
					child [fill=none] {edge from parent[draw=none]}
					child {
						child
						child 
						child
					}
				};
		\tikzstyle{level 1}=[level distance=.08in, color = blue]
		\tikzstyle{level 2}=[level distance=.08in, sibling distance=.08in, color = blue]
		\tikzstyle{level 3}=[level distance=.08in, sibling distance=.08in, color = blue]
		\tikzstyle{level 4}=[level distance=.08in, sibling distance=.08in, color = blue]	
			\node[fill=none] at (2,3) {} [grow'=up]
				child {
					child 
					child 
					child {
						child
						child [fill=none] {edge from parent[draw=none]} 
						child
					}
				};
\tikzstyle{level 1}=[level distance=.08in, color = red]
		\tikzstyle{level 2}=[level distance=.08in, sibling distance=.08in, color = red]
		\tikzstyle{level 3}=[level distance=.08in, sibling distance=.08in, color = red]
		\tikzstyle{level 4}=[level distance=.08in, sibling distance=.08in, color = red]	
			\node[fill=none] at (0,4) {} [grow'=up]
				child {
					child 
					child 
					child 
					child
				};
\end{tikzpicture}
\caption{$\K^4$ with its cells indexed by trees}
\end{figure}

%% file: SECTION3-thomspectrum.tex
The goal of this section is to obtain the Moore spectrum $\M{p}{i}$ as a Thom spectrum associated to a  $p$-adic spherical bundle. We begin with a nontechnical description of a general construction of Thom spectra  following \cite{ABGHR}, avoiding the hard technical work of \cite{ABGHR} and presenting only the gist.

To begin with, we would like to point out that the right adjoint of the  $\sph[ \ ]: \Top \to \Sp$ is not the `zeroth space' functor as one would expect because of  our choice of $\Sp$. However, if the category $(\Top,\ast, \times)$ is replaced with the category of $\ast$-modules $(\TTop, *, \times_{\mathcal{L}})$ (which is Quillen equivalent to $\Top$), then the right adjoint  is naturally weakly equivalent to the zeroth space functor. The \cite{EKMM} category of $\sph$-module admits a loop-suspension adjunction  
\begin{equation} \label{eqn:spectraadj}
\sph[ \ \ ]:\TTop \ \substack{\longrightarrow \\[-0.7ex] \longleftarrow} \ \Sp: \Omega^{\infty}.
\end{equation}
A detailed exposition can be found in  \cite[Section~$3.1$]{ABGHR}. A detailed discussion can be found in \cite{Lind} (also see \cite[Section~$3$]{ABGHR}).

For a ring spectrum $\R$, define $\GL(\R)$ to be the pullback   
\[ 
\xymatrix{
\GL(\R) \ar[r]\ar[d]  & \Omega^{\infty}\R \ar[d] \\
\pi_0(\R)^{\times}  \ar[r]  & \pi_0(\R).
}
\]
 In other words, $\GL(\R)$ is the collection of components of $\Omega^{\infty} \R$ that correspond to the units of $\pi_0(\R)$. Similarly,  for any subgroup $\H$ of $\pi_0(\R)^{\times}$, define $\H(\R)$ as the pullback
\[ 
\xymatrix{
\H(\R) \ar[r]\ar[d]  & \Omega^{\infty}\R \ar[d] \\
\H  \ar[r]  & \pi_0(\R). 
}
\]
When $\H$ is the trivial subgroup, $\H(\R)$ is denoted by $\SL(\R)$ in the literature. If $\R$ is an $\A_{\infty}$-ring spectrum (over the linear isometry operad), then $\H(\R)$ is a group-like monoid in $\TTop$. Therefore, one can perform the two-sided bar-construction 
\begin{equation} \label{eqn:classify}
 \mr{BH}(\R) = \mr{B}_{\times_{\mathcal{L}}}(*,\H(\R)^c,*),
\end{equation}
where $\H(\R)^c$ denotes the cofibrant replacement of $\H(\R)$.  $\mr{BH}(\R)$ is the `classifying space' for the principal $\H(\R)^c$-bundle. If $\R$ is an $\Einfty$-ring spectrum then $\H(\R)$ is a grouplike commutative monoid and is represented by a spectrum, call it $\mr{h}(\R)$.
\begin{notn}
  The conventional notations for $\mr{h}(\R)$ when $\H = \pi_0(\R)^{\times}$ and $\H = \{1\}$ are $\mr{gl}_1(\R)$ and $\mr{sl}_1(\R)$, respectively. We will adhere to the conventional notations for these special cases.
\end{notn}
Now, we explain the  construction of the Thom spectrum associated to a map in $\TTop$
\[ f: X \to \mr{BH}(\R).\]
 Let $\mr{P}$ be the associated principal $\H(\R)^c$-bundle and $\mr{P}'$ be its cofibrant replacement as a right $\H(\R)^c$-module. The spectrum $\sph[\mr{P}']$ admits a right $\sph[\H(\R)^c]$-module structure. On the other hand, there is a natural map 
\[ \gamma: \sph[\H(\R)^c] \longrightarrow \sph[\H(\R)] \longrightarrow \R\]   
which makes $\R$ into a left $\sph[\H(\R)^c]$-module.  The \emph{Thom spectrum} associated to the map $f: X \to \mr{BH}(\R)$ is the derived smash product
\[ {\sf M}(f) := \sph[P'] \underset{\sph[\H(\R)^c]}{\sma}\R.\]
 The construction of Thom spectra is a functor 
\begin{equation} \label{eqn:thomfunctor}
\begin{tikzcd} 
 {\sf M} : \TTop/_{\mr{BH}(\R)} \rar & \Sp_{\R},
 \end{tikzcd}
 \end{equation}
where $\Sp_{\R}$ is the category of $\R$-modules. 
\begin{notn} To distinguish stable homotopy groups from the unstable ones, we denote the  $n$-th unstable homotopy group functor  by
$$\pi_n^u: \Top_* \to \text{Groups}.$$   
The functor $\pi_n^u$ can be extended to $\TTop_*$ via the Quillen adjunction, where $\TTop_*$ denotes the based category of $\ast$-modules.  
\end{notn}
Note that  $\pi_0^{u}(\H(\R)) \iso \H$  by construction and loop-suspension adjunction implies 
\[ \pi_n^{u}(\H(\R)) \cong \pi_n^{u}(\Omega^{\infty}\R) \cong \pi_n(\R)  \]
for $n\geq 1$. Therefore $\pi_1(\mr{BH}(\R)) = \H$ and there is an isomorphism 
\[ \Theta: \pi_n^u(\mr{BH}(\R)) \overset{\iso} \longrightarrow \pi_{n-1}(\R).\]
Understanding the construction of the isomorphism $\Theta$ is crucial in the proof of \Cref{lem:thom}.

 By adapting Steenrod's classification theorem in our settings we get
\[ \pi_n^{u}(\mr{BH}(\R)) \iso \Prin_{\H(\R)^c}(\tsph^n), \]
where $\Prin_{\H(\R)^c}(\tsph^n)$ denotes the isomorphism classes of principal $H(\R)^c$-bundles over a cofibrant $n$-sphere in $\TTop$. Let $\tdisk^n_+$ and $\tdisk^n_-$ be the northern and the southern hemispheres respectively of the $n$-sphere
   \[\tsph^n = \tdisk^n_+ \cup_{\tsph^{n-1}}\tdisk^n_- .\]
Fix a basepoint $x_0$ of $\tsph^n$  on the equator. Let  $\mr{P}(f)$ denote the  principal $\H(\R)^c$ bundle $\mr{P}$ over $\tsph^n$ associated to a map 
\[ f: \tsph^n \to \mr{BH}(\R)\]
The restriction of $\mr{P}(f)$ to $\tdisk^n_+$ and $\tdisk^n_-$, are trivial bundles as the base spaces are contractible. Thus $\mr{P}(f)$ is the pushout in the diagram 
\begin{equation} \label{eqn:principal}
 \xymatrix{
\tsph^{n-1} \times_{\mathcal{L}} \H(\R)^c \ar[rr]^{\tau} \ar[d]_{i} && \tdisk^{n}_{+} \times_{\mathcal{L}} \H(\R)^c \ar[d] \\
\tdisk^{n}_{-} \times_{\mathcal{L}} \H(\R)^c \ar[rr] && \mr{P}(f)
}
\end{equation}
 where $i(x,g) = (x,g)$,  $\tau(x,g) = (x, \theta_f(x)g)$ and $\theta_f:\tsph^{n-1} \to \H(\R)^c$ is the \emph{clutching function} defined on the equator sending 
 \begin{equation} \label{eqn:bp}
 x_0 \mapsto 1_{\H(\R)^c}.
 \end{equation}
 Assigning each principal bundle  over $S^n$ its clutching function produces the isomorphism $\Theta$. 


\begin{lem} \label{lem:thom} Let $\alpha \in \pi_{n-1}(\R)$, $f_{\alpha}: \tsph^{n} \to \mr{BH}(\R)$ so that  $[f_{\alpha}] = \Theta^{-1}(\alpha) \in\pi_n^u(\mr{BH}(\R))$ and 
${\sf M}(f)$ denote the corresponding Thom spectrum. Let $\alpha_{\R}$ denote the map 
\[ 
\begin{tikzcd}
\alpha_{\R}: \Sigma^{n-1}\R \rar & \R \sma\R  \rar & \R.
\end{tikzcd}
\]
Then 
\[ 
{\sf M}(f) \simeq \left\lbrace \begin{array}{ccc} 
\mr{C}(\alpha_{\R}) & \text{when $n \geq 2$,  }\\
\mr{C}(\mathbbm{1}_{\R } - \alpha_{\R}) & \text{when $n =1$.  }
\end{array} \right.
\]

 ${\sf M}(f)$ associated to a map $f: \tsph^{n} \to \mr{BH}(\R)$, where $[f] = \Theta^{-1}(\alpha) \in\pi_n(\mr{BH}(\R))$
\end{lem}

\begin{proof} Since  $[f_{\alpha}] = \Theta^{-1}(\alpha)$, the clutching function   $\theta_{f_{\alpha}}$ must  belong to the class of  \[ \alpha \in \pi_{n-1}^u(H(\R)^c).\] 
Now we apply the functor  $\sph[\underline{\hspace{5pt}}] \underset{\sph[H(\R)^c]}{\sma} \R$ to the diagram in \Cref{eqn:principal}.  When $n \geq 2$, we get ${\sf M}(f_{\alpha})$ as the pushout 
\[ 
\xymatrix{
\Sigma^{n-1}\R \ar[d]_{0}\ar[r]^{\alpha_\R} & \R \ar[d]\\
\R \ar[r] & {\sf M}(f_{\alpha}).
}
\] 
Now consider the case when  $n=1$. Because the basepoint maps to the unit component of $H(\R)$ (see~\Cref{eqn:bp}), we get ${\sf M}(f_{\alpha})$ as the pushout of the diagram 
\[ 
\xymatrix{
\R \ar[d]_{\mathbbm{1}_{\R}}\ar[r]^{\alpha_\R} & \R \ar[d]\\
\R \ar[r] & {\sf M}(f_{\alpha}).
}
\] 
The result follows from the above diagrams.
 \end{proof}
 
The $p$-completion of the sphere spectrum $\hat{\sph}_p$ is an $\Einfty$-ring spectrum. The fundamental group of $\BGL(\hat{\sph}_p)$ is 
\[ \padic_p^{\times} \cong \left\lbrace 
\begin{array}{ccccc}
\mathbb{Z}/(p-1) \times \padic_p & \text{when $p$ is an odd prime}\\
\mathbb{Z}/2 \times \padic_2 & \text{when $p=2$.}
\end{array}
\right.\]
Let \[
\begin{tikzcd}
 \hat{f}_{p,i}: \mr{S}^1 \rar & \BGL(\hat{\sph}_p)
 \end{tikzcd} \]
 denote a map representing the class $1 + p^i u \in \padic_p^{\times}$, where $u$ is a unit. The following result is a straightforward consequence of \Cref{lem:thom} 
\begin{cor}\label{cor:Moorethom}
The Thom spectrum ${\sf M}(\hat{f}_{p,i})$ is homotopy equivalent to $\M{p}{i}$, the $i$-th Moore spectrum at the prime $p$. 
\end{cor}

%% file: SECTION4-Maintheorem.tex
Realizing the Moore spectrum $\M{p}{i}$ as a Thom spectrum has a distinct advantage. Note that  the domain and the codomain of the associated map $\hat{f}_{p,i}$  are $\Ainfty$-spaces, in fact both are  $\Einfty$-spaces. It is a consequence of a theorem, originally due to Lewis \cite[$\mathsection$ IX]{LMS}, that if $\hat{f}_{p,i}$ is an $\A_n$-map then $\M{p}{i}$ admits an $\A_n$-structure.

\begin{thm}[Lewis] \label{thm:Anthom}
 Let $X \in \Top$ be an  $\A_n$-algebra  and admits an $\A_n$-map 
\[ f: X \to \BGL(\hat{\sph}_p)\]
then the Thom spectrum ${\sf M}(f)$ inherits a unital $A_n$-structure.
\end{thm}
Here is a brief  explanation of the proof of Theorem~\ref{thm:Anthom}. The fact that $f$ is an $\A_n$-map means that we have a commutative diagram 
\[ 
\xymatrix{
\K^i \times X^{\times i} \ar[rr]^{\K^i \times f^{\times i}} \ar[d]^{\mu_i} && \K^i \times \BGL(\hat{\sph}_p) \ar[d] \\
X \ar[rr]_{f} && \BGL(\hat{\sph}_p)
}
\]
for $0 \leq i \leq n$. The Thom spectrum associated to the map $\mu_i \circ (\K^i \times f^{\times i})$ is $\K^i_+ \sma {\sf M}(f)^{\sma i}$. Since the construction of Thom spectrum is functorial \Cref{eqn:thomfunctor}, we get a map 
 \[
 \begin{tikzcd}
 \overline{\mu}_i: \K^i_+ \sma {\sf M}(f)^{\sma i} \rar & {\sf M}(f)
 \end{tikzcd}
  \] 
  for $0 \leq i \leq n$ . It can be shown that the maps $\overline{\mu}_i$ for $0 \leq i \leq n$ induce an  $\A_n$-algebra structure on ${\sf M}(f)$. 

\begin{rmk}
Lewis showed that the Thom spectrum associated to an $\mathcal{O}$-map
\[ f: X \to \SL(\sph)\]
where $\mathcal{O}$ is any topological operad,  admits an $\mathcal{O}$-algebra structure. Lewis worked in the  Lewis-May-Steinberger \cite{LMS} category of spectra. In \cite{ABGHR}, the authors generalized Lewis' result in the \cite{EKMM} category of $\sph$-modules, where they replace $\SL(\sph)$ with general grouplike objects such as $\SL(\R)$ and $\GL(\R)$ for arbitrary  $\Einfty$-ring spectrum $\R$, but restrict themselves to $\Einfty$-structures only. However, combining the work  of Lewis in \cite[$\mathsection$ IX]{LMS} and \cite{ABGHR} one can obtain \Cref{thm:Anthom}.
\end{rmk}

\begin{cor} $\M{p}{i}$ admits an $\A_n$-algebra structure if there exists  a map  
\[
\begin{tikzcd}
 \hat{f}_{p,i}: \tsph^1 \rar & \BGL(\hat{\sph}_p)
 \end{tikzcd} \]
 in the homotopy class $1 + p^iu \in \pi_1(\BGL(\hat{\sph}_p)) \cong \padic_p^{\times}$ (where $u \in \padic_p^{\times}$) which can be extended to an $\A_n$-map. 
\end{cor}

 Stashef's work \Cref{thm:barasso}, along with Boardman and Vogt's rigidification result \Cref{thm:Anfix}, implies that $\hat{f}_{p,i}$ can be extended to an $\A_n$-map if and only if $\Sigma \hat{f}_{p,i}$ extends to a map $\hat{f}_{p,i}^{(n)}$ so that the diagram 
 \begin{equation} \label{support1}
 \begin{tikzcd}
 \Sigma \tsph^1 \arrow[d, hook] \rar["\hat{f}_{p,i}"] & \Sigma \BGL(\hat{\sph}_p) \arrow[d,hook] \rar["\iota_1"] & \mr{B}^2\GL(\hat{\sph}_p) \\
 \CP^n \arrow[r,  dashed, " \hat{f}_{p,i}^{(n)} "'] &  \mr{B}_n(\BGL(\hat{\sph}_p)) \arrow[ru, "\iota_n"']
 \end{tikzcd}
 \end{equation}
 commutes. Note that the exists of $\hat{f}_{p,i}^{(n)} $ is equivalent to solving the lifting problem 
 \begin{equation} \label{support2}
 \begin{tikzcd}
 \Sigma \tsph^1 \arrow[d, hook] \arrow[rr, "\iota_1 \circ \hat{f}_{p,i} "] && \mr{B}^2\GL(\hat{\sph}_p) \\
 \CP^n \arrow[urr, dashed, "\tilde{f}_{p,i}^{(n)}"']
 \end{tikzcd}
 \end{equation}
in the homotopy category of $\TTop_*$.  If $\tilde{f}_{p,i}^{(n)}$ exists and \Cref{support2} commutes up to homotopy, then $\mr{CW}$-approximation theorem implies that there exists  $\hat{f}_{p,i}^{(n)} $  so that \Cref{support1} commutes up to homotopy. Choosing the 
 map $\Sigma\tsph^1 \hookrightarrow \CP^n$ to be a cofibration and using homotopy extension property one can make sure that \Cref{support1} commutes on the nose.

 Since, $\GL(\hat{\sph}_p)$ is an infinite loop space, a lift in \Cref{support2} is equivalent to the stable lifting 
 \begin{equation} \label{support3}
 \begin{tikzcd}
 \sph \arrow[d, hook] \arrow[rr, "1 + p^i u "] && \mr{gl}_1(\hat{\sph}_p) \\
 \Sigma^{-2}\CP^n \arrow[urr, dashed, "f_{p,i}^{(n)}"']
 \end{tikzcd}
 \end{equation}
 in the homotopy category of $\Sp$, which can be studied using the Atiyah-Hirzebruch spectral sequence (AHSS)
 \begin{equation} \label{eqn:AHSS}
 \mr{E}_2^{s,t} := \H^s(\Sigma^{-2}\CP^n; \pi_t (\mr{gl}_1(\hat{\sph}_p))) \Rightarrow [ \Sigma^{-2} \CP^n, \mr{gl}_1(\hat{\sph}_p)]_{t-s}
 \end{equation}
 Note that $1 + p$ is the generator of the subgroup $\padic_p \subset  \pi_0(\mr{gl}_1(\hat{\sph}_p))$ and 
 \[p^i \cdot (1+p) = (1+p)^{p^i} = \left\lbrace
 \begin{array}{cccc}
 1+ p^i u & \text{ if $p$ is an odd prime} \\
 1+ 2^{i+1}u & \text{when $p=2$}
 \end{array}
  \right. \] 
 where $u $ is a unit. If $p^i \cdot (1+ p) \in \mr{E}^{0,0}$ survives the AHSS, then we can conclude that $\M{p}{i}$ admits an $\A_n$-algebra structure when $p$ is odd and $\M{2}{i}$ admits an $\A_n$-algebra structure when $p=2$. Since the cohomology of  $\CP^n$ 
 \[ \H^*(\CP^n; \mathbb{Z}) \iso \mathbb{Z}[x_2]/(x_2^{n+1})\] is zero in odd dimensions and $\mathbb{Z}$ in even dimension upto $2n$, the target of the $d_{r}$-differentials 
 \[
 \begin{tikzcd} 
 d_r: \E^{0,0}_r \rar & \E^{r,r-1}_r
 \end{tikzcd} 
 \]
 is zero when $r$ is odd and the finite torsion group 
 \[ \pi_{r-1} (\mr{gl}_1(\hat{\sph}_p)) \cong \pi_{r-1}(\hat{\sph}_p) \] 
 when $r$ is even. This, and and the fact that \Cref{eqn:AHSS} only has finitely many differentials can be used to estimate the value of $i$ for which $p^i \cdot (1+p)$ survives the spectral sequence. We will improve this estimate and conclude the proof \Cref{main:est1} through the following lemma. 
 \begin{lem} \label{lem:refineest} In the $\AHSS$ \Cref{eqn:AHSS}, $p^i \cdot (1+ p) \in \E^{0,0}$ is a permanent cycle if $i > o_p(n)$ where $o_p(n)$ is the function defined in \Cref{main:est1}.   
 \end{lem}
 \begin{proof} Differentials in $\AHSS$ are $\mathbb{Z}$-linear, therefore we have 
 \begin{equation} \label{eqn:1}
  d_{2r}( p \cdot x) = p \cdot d_{2r}(x).
  \end{equation}
  Also an $\AHSS$ $\H^*(X; \pi_*(Y)) \Rightarrow [X, Y]_*$ is natural in the variable $X$. Since multiplication by $p$ map on $\Sigma^{-2}\CP^n$ induces multiplication by  $p^{r+1}$ on $\H^r(\Sigma^{-2}\CP^n; \mathbb{Z})$ we also have 
 \begin{equation} \label{eqn:2}
  d_{2r}( p \cdot x) = p^{r+1} \cdot d_{2r}(x).
  \end{equation}
 Combining \Cref{eqn:1}  and \Cref{eqn:2}, we get 
 \[ 
 d_{2r}(p \cdot x - p^{r+1} \cdot x) = p(1-p^r) \cdot d_{2r}(x) = 0 
 \]
  Since, $1 - p^r$ is a unit in $\padic_p$, $d_{2r}(p \cdot x) = p \cdot d_{2r}(x) = 0$. Thus, if $x$ supports a $d_{2r}$-differential then $p \cdot x$ is a $d_{2r}$-cycle. The result follows from a simple inductive argument.  
 \end{proof}
 
 One could hope to improve \Cref{main:est1}, if the map $1 + p^i u$ factors through an infinite-loop space map 
 \begin{equation} \label{eqn:J}
  \mathcal{J}:\widetilde{\mr{im(J)}} \to \GL(\hat{\sph}_p) 
  \end{equation}
 where $\widetilde{\mr{im(J)}}$ is the infinite loop space associated to a hypothetical connective `image of $\mr{J}$'  spectrum whose zeroth homotopy group is $\padic_p^{\times}$. The zeroth homotopy group of the classical image of $\mr{J}$ spectrum is $\mathbb{Z}/2 \cong \pi_0(\sph)^{\times}$ and it is known that the associated infinite-loop space splits of the unit component of the integral sphere after completing at an odd prime. Whether we can construct this hypothetical `image of $\mr{J}$' spectrum and whether there exists an infinite loop space map 
 $\mathcal{J}$ (as in \Cref{eqn:J}) is the subject of study in the joint work \cite{BK} with N.Kitchloo. We expect the answers to be affirmative, and hence, we end this paper with the following conjecture. 
 \begin{conj} \label{conj:J} For an odd prime $p$, the obstructions to $\Ainfty$-structure on $\M{p}{i}$ lies in the image of $\mr{J}$-part (the chromatic layer $1$) of the homotopy groups of $\M{p}{i}$. 
 \end{conj}